\documentclass[12pt,leqno]{amsart}
\usepackage[dvips]{graphics}
\usepackage{amssymb,hyperref}
\usepackage{color,comment}

\numberwithin{equation}{section}

\newtheorem{thm}{Theorem}[section]

\newtheorem{lem}[thm]{Lemma}
\newtheorem{cor}[thm]{Corollary}
\newtheorem{prop}[thm]{Proposition}

\theoremstyle{remark}
\newtheorem{rem}[thm]{Remark}
\newtheorem{ex}[thm]{Example}

\newcommand{\tref}[1]{Theorem~\ref{#1}}
\newcommand{\cref}[1]{Korollar~\ref{#1}}

\newcommand{\diam}{\mathrm{diam}}

\newcommand{\R}{\mathbb{R}}
\newcommand{\C}{\mathbb{C}}

\newcommand{\area}{\mathrm{area}}

\begin{document}
%\tableofcontents
\pagebreak
%\bibliographystyle{alpha}

%\pagenumbering{roman}

\title{On conformal planes of finite area}

\thanks{
	A. L. was partially supported by the DFG grants   SFB TRR 191 and SPP 2026.}

\author{Alexander Lytchak}

\address
{Institute for Algebra and Geometry\\KIT\\ Englerstr. 2 \\ 76131 Karlsruhe, Germany}
\email{alexander.lytchak@kit.edu}

%\subjclass
%[2010]{53C20, 53C21, 53C23}

%\date{\today}

%\thanks{The second author was partially supported by Swiss National Science Foundation Grant 153599}

\begin{abstract}
	We discuss solutions of several  questions concerning the geometry of conformal planes.
\end{abstract}

%\author{Alexander Lytchak}
%\address{Mathematisches Institut\\ Universit\"at Bonn\\
%Wegelerstrasse 10, 53115 Bonn, Germany\\}
%\email{lytchak\@@math.uni-bonn.de}

%\subjclass{53C20,  52B99}
%\footnotetext[1]{}

\maketitle

\section{Introduction}
\subsection{Applications} Recently, the Liouville equation
\begin{equation}   \label{eq: Liouv}
\Delta u + e ^{2u} =0 \,,
\end{equation}
and its (super-) solutions   on $\R^2$  were investigated in a series of work \cite{Liouv}, \cite{Liouvrig}, \cite{Erem}, see also \cite{ChenLi}, \cite{ChouWan}. Interesting facts on the geometry of the  corresponding conformal planes 
$$X^u= (\R^2, e^{2u} \cdot \delta _{Eucl}) $$
were proven  and the authors formulated  several related questions. 

Solutions  of  \eqref{eq: Liouv} correspond  to conformal planes of constant curvature $1$ and   are  closely related to  some meromorphic functions on $\C$. Complex analysis can been successfully used to study the solutions and arising geometries \cite{Liouvrig}, \cite{Erem}.  For supersolutions of \eqref{eq: Liouv}, thus for conformal metrics on the plane of curvature $\geq 1$,  complex analysis does not seem to be such an appropriate tool.

The theory of surfaces with integral curvature bounds in the sense of Alexandrov,  see \cite{AZ}, \cite{Reshetnyak-GeomIV},  \cite{Troyanov}  turns out to be more  helpful, especially, for questions concerning  conformal planes of bounded total area and curvature. This  approach implies   the following solutions to four  questions formulated in \cite{Liouv} and \cite{Smooth}.   

\begin{prop} \label{prop: 1}
For  a smooth $u:\R^2 \to \R$ satisfying 
\begin{equation}   \label{eq: Liouv+}
\Delta u + e ^{2u} \leq 0 \,,
\end{equation}
 let the conformal plane $X^u$ have finite area. Then the diameter $\diam (X^u)$ of the plane $X^u$ can be any number in the interval  $(0,2\pi)$. 
\end{prop}

In \cite[Theorem 1.4]{Liouv}, it was proved that  \eqref{eq: Liouv+} implies $\diam (X^u) \leq 2\pi$, and 
\cite[Question 8.2]{Liouv} asks if the inequality $\diam (X^u) \leq \pi$ holds.

\begin{prop} \label{prop: area}
For  a smooth $u:\R^2 \to \R$ satisfying  \eqref{eq: Liouv+},
the area of the conformal plane $X^u$ can be infinite or any positive real number. 
\end{prop}

On contrary, for solutions $u$ of \eqref{eq: Liouv} 
 the conformal planes $X^u$ have area   $4\pi$ or infinity, \cite{Liouv}.
 It has been asked in \cite[Question 8.3]{Liouv}, whether the upper bound of $4\pi$ is valid for all  conformal planes $X^u$ 
of finite area  corresponding to solutions of  \eqref{eq: Liouv+}. The above result has been  independently  observed by Alexandre Eremenko.

As a  consequence, we deduce a negative answer to another question formulated in \cite[Question 8.7]{Liouv},  see Corollary \ref{prop: 3} below.

\subsection{From the  sphere to conformal planes}
The above results    are easy consequences of known theorems on singular  metrics on $\mathbb S^2$ with bounded integral curvature and of a simple 
relation between   conformal planes and  conformal spheres, which we are going to explain now.

By the uniformization theorem, any Riemannian metric on $\R^2$  is either conformally equivalent to the disc or to the plane. While it is easy to construct many (non-complete) Riemannian metrics on $\R^2$ with prescribed curvature properties, (for instance, with constant curvature 1),  it seems  difficult to verify that such a synthetically constructed metric is a conformal plane.  A criterion of conformality is provided by the special case of a classical result of Cheng--Yau \cite[Corollary 1]{CY}: If a \emph{complete} Riemannian manifold   $X$ homeomorphic to $\R^2$ has  at most \emph{quadratic area growth}
then $X$ is a conformal plane. In particular,  all \emph{complete} Riemannian metrics of \emph{finite area} on $\R^2$ are conformal planes.

 An easy criterion for \emph{non-complete} planes, sufficient for the Propositions stated above, is the following one. 

\begin{prop} \label{prop: unif}
Let $X$ be a Riemannian manifold homeomorphic to the plane and of finite area.
%If $X$ is complete the $X$ is conformally equivalent to the plane.
Assume that the completion
$\hat X$ of $X$ is homeomorphic to $\mathbb S^2$ and that $\hat X \setminus X$ has just one point $p$.  If the area of metric balls $B_r(p)$ in $\hat X$ around  $p$ grows at most quadratically, $$\liminf _{r\to 0} \frac {\area (B_{r}(p))} {r^2}  < \infty \;,$$
then $X$ is conformally equivalent to the plane.
\end{prop}

It might be possible to deduce Proposition \ref{prop: unif} from the theorem by Cheng--Yau mentioned above, applying a conformal change of the metric, which resembles  the \emph{inversion} at the point $p$. Instead, we observe that Proposition \ref{prop: unif} is a consequence of a very general uniformization theorem in metric geometry \cite{BK}, \cite{Raj}, \cite{LW}, \cite{NR}.

\begin{rem}
Some assumption in Proposition \ref{prop: unif} on a neighborhood of $p$ in $\hat X$ is needed, as the following easy example demonstrates: Consider the unit Euclidean disc with the conformal factor 
$f(z)=(1-|z|^2)$.  The completion $\hat X$ of this conformal disc $X$ has finite area,  is  homeomorphic to $\mathbb S^2$, and $\hat X\setminus X$ has just one point. 
\end{rem}

Thus, in order to construct conformal planes with prescribed properties as in Propositions \ref{prop: 1}, \ref{prop: area}, it suffices to construct metrics on the sphere with one singularity $p$  and prescribed  geometric properties outside the singularity. We construct such a piecewise spherical metric with only 3 vertices, such that the total angle at just one of these vertices (the singularity $p$) is larger than $2\pi$. Note that all such metics are classified \cite{Erem1}, \cite{Panov}.  Smoothing the metric at the  singularites with angles smaller than $2\pi$,  we obtain the desired examples.  These examples have bounded integral curvature in the sense of Alexandrov, \cite{AZ},  \cite{Reshetnyak-GeomIV}, \cite{Troyanov};  more classical uniformization theorems, \cite{Troyanov}, imply  the conclusion of Proposition \ref{prop: unif} in this case.

\subsection{Completions of conformal planes}
A  partial converse to  Proposition \ref{prop: unif} is
 essentially contained 
in the proof of \cite[Theorem 1.4]{Liouv}:

\begin{lem} \label{prop: incompl}
Let  the  conformal plane  $X=X^u$ have finite area and let $\hat X$ denote the completion of $X$. Then $\hat X\setminus X$ has at most one point. 
%
%Then either $X$ is complete or 
%the  completion $\hat X$ of $X$ is homeomorphic to $\mathbb S^2$ and $\hat X\setminus X$ has just one point.
\end{lem}

Thus, either $X$ is complete or $\hat X \setminus X$ has exactly one point $p$.
In the latter case, the space $\hat X$ can display a rather wild   behavior
near $p$. For instance, it may not  be locally compact around $p$, see Example \ref{ex: non} below.  Even if  $X$ has curvature larger than $1$ and  $\hat X$ is compact, thus homeomorphic to $\mathbb S^2$, the geometry around  $p$ can be rather wild,
 see Example \ref{ex: Romney} below.

 The geometry of the completion $\hat X$ at the \emph{singular point} $\hat X\setminus X$  turns out to be much tamer  if the curvature on $X$ is assumed to be integrable.

Recall first that the  Hausdorff (=canonical Riemannian)  area  $\mathcal H^2$ on the conformal plane $X^u$ is the multiple  $e^{2u} \cdot \mathcal L^2 _{\R ^2}$ of the Lebesgue area $\mathcal L^2$.    Thus the \emph{total area} of $X^u$ equals $\mathcal A(X^u)= \int _{\R^2} e^{2u}$.

The curvature of the conformal plane $X^u$ equals $K=e^{-2u} \cdot \Delta u$.  Thus, the 
(integral) 
boundedness of the curvature of $X^u$, is  the analytic assumptions
$\Delta u \in L^{\infty} (\R^2)$  ($\Delta u \in L^{1} (\R^2)$).
If  $\Delta (u)\in L^1 (\R^2)$ then
$$\mathcal K(X^u):=\int _{\R^2} \Delta u  \,d\mathcal L^2 _{\R^2}= \int _{X^u}K  \, d\mathcal H^2 _X$$ 
is called the total curvature of $X^u$.

Most parts of the next result are scattered through the literature:

\begin{thm} \label{thm: main}
Let $X=X^u$ be a conformal plane of finite area $\mathcal A(X)$ and finite total curvature 
$\mathcal K(X)$.  Then  $\mathcal K (X) \geq  2\pi$. 
If $\mathcal K(X) >2\pi$ then $X^u$ is not complete.

If $X$ is not complete then the completion $\hat X$ is a sphere which has bounded integral curvature in the sense of Alexandrov.

Upon a conformal  identification of $\R^2$ with $\mathbb S^2\setminus \{p \}$, the function 
$u$ defines a $\delta$-subharmonic function on $\mathbb S^2$, in the complete and in the non-complete case.
\end{thm}

Recall that  a function is called $\delta$-subharmonic   if locally around any point it can  represented as  a difference of two subharmonic functions.

The  theory of surfaces with integral curvature bounds implies that in the non-complete case, the area growth is at most quadratic 
 at the point $p=\hat X \setminus X$.
Moreover,
limes inferior arising in Proposition \ref{prop: unif} is a limit and equals 
$\frac {\mathcal K(X)} 2  -\pi$, see Section \ref{subsec: int}.

\subsection{Uniformly bounded curvature}

A final application  answers the question investigated in \cite{Smooth} and relates this question to the theory of manifolds with both-sided cuvature bounds, \cite{Ber-Nik}.  Slightly weaker results have been obtained in \cite{Smooth} by direct methods.

\begin{prop} \label{prop: final}
Assume that the plane $X=X^u$ has finite area and that the total  curvature $\mathcal K(X)$
equals $4\pi$. If the curvature $K$ 
  of $X$ is uniformly bounded 
 then the completion
$\hat X$ of $X$ is a Riemannian manifold conformally equivalent to the round sphere $\mathbb S^2$.
For the conformal factor $e^{2 \hat u}$, the function $\hat u$ is of
 class $\mathcal C^{1,\alpha}$  on $\mathbb S^2$,  for every $\alpha <1$.

Even if  the curvature $K$ is continuous  on  $\hat X$, the function $\hat u $ does not need to be $\mathcal C^{1,1}$.  If $K$ is $\beta$-Hoelder on $\mathbb S^2$ then $\hat u$ is $\mathcal C^{2,\beta}$. 
\end{prop}

\subsection{Acknowledgements}  I thank Qinfeng Li for helpful communication giving rise to this note.  I am grateful to Matthew Romney
for details about the example appearing in \cite{CR}, \cite{Romney} and for the reference  \cite{Itoh}.  I would like to thank Anton Petrunin for helpful discussions and to  Dima Panov for sharing   his examples of large piecewise spherical spheres, which has simplified the proof of Proposition \ref{prop: 1}. For useful comments I thank Paul Creutz, Qinfeng Li and Matthew Romney.

\section{From the sphere to the plane}
\subsection{One-point complements in spheres}
In the proof of Proposition \ref{prop: unif} below, we are going to freely use the vocabulary of
metric geometry.  We refer to \cite{NR} for the definitions and properties, in particular for the notion of \emph{weak conformality}.

\begin{proof}[Proof of Proposition \ref{prop: unif}]
By assumption, we have a geodesic metric space $\hat X$, homeomorphic to $\mathbb S^2$
and a point $p\in \hat X$ such that $X=\hat X\setminus \{p \}$  has a smooth Riemannian metric.
By assumption, the area growth at $p$ is at most quadratic.  In particular, $\hat X$
has finite 2-dimensional  Hausdorff measure.

By \cite[Theorem 1.3]{NR}, there exists a weakly quasiconformal map $h:\mathbb S^2 \to \hat X$ from the round sphere $\mathbb S^2$.

The area growth assumption implies that  $h$ is a homeomorphism, \cite[Theorem 7.4]{NR}. The map $h$ restricts to a weakly quasiconformal map from $\mathbb S^2\setminus 
h^{-1} (p) \to X$.  Since $h^{-1} (p)$ is a singleton, $\mathbb S^2\setminus  h^{-1} (p)$
is conformally equivalent to $\R^2$.   Therefore, we have a weakly quasiconformal map between smooth Riemannian manifolds
$\hat h: \R^2 \to X$.  If $X$ were a conformal disc, we would obtain a weakly quasiconformal homeomorphism from $\R^2$ to the disc $D$. Such a homeomorphism  cannot exist, see, for instance,  \cite[p. 2-4]{quasi}. 
\end{proof}

Assuming that $\hat X$ has bounded integral curvature in the sense of Alexandrov, \cite{AZ}, \cite{Reshetnyak-GeomIV}, \cite{Troyanov}, a shorter proof of Proposition \ref{prop: unif}
is possible.  Indeed, in this case,  the uniformization theorem, \cite[Section 7]{Troyanov} states that the metric on
$\hat X$ is defined as $e^v\cdot \delta _{\mathbb S^2}$, where the function $v$ in the conformal factor is $\delta$-subharmonic on $\mathbb S^2$.  This directly describes $X =\hat X\setminus \{p\}$ as conformally changed 
$\mathbb S ^2$ without  point.

\subsection{Some examples of conformal planes} We are going to prove Proposition \ref{prop: 1} and Proposition \ref{prop: area}.     Observe first, that rescaling the metric by a positive constant  $\lambda \leq 1$ provides again a metric in the same class
(curvature at least 1, finite area).  Thus, it suffices to find conformal planes of curvature $\geq 1$ and arbitrary large finite area, respectively, finite area and diameter arbitrary close to $2\pi$.

Consider a piecewise spherical metric on $\mathbb S^2$ such that the total angle is larger than $2\pi$ in at most one singularity $p$. 
In the arising metric space $Y$ the curvature is constant $1$ outside $p$ and finitely many further points $p_1,..,p_n$.  Around any point $p_i$ the metric  is a spherical cone metric over a circle  of length  less than $2\pi$. 
The metric around $p_i$ can be smoothened in an arbitrary small neighborhood, such that the arising metric is smooth and  has curvature $\geq 1$,
 \cite[Lemma 2.4]{Itoh}. Moreover, by construction, the new smooth metric has almost the same diameter and area as the original one.

     Performing this operation around every vertex $p_1,...,p_n$, we obtain a metric space $Y_{\varepsilon}$ homeomorphic to $\mathbb S^2$, such that $X:=Y_{\varepsilon} \setminus p$ is a smooth Riemannian manifold of curvature $\geq 1$.  This manifold $X$ is a conformal plane 
by Proposition \ref{prop: unif}; it  has finite area  and diameter arbitrary close (by the choice of  $\varepsilon$) to the area  and the diameter  of $Y$.

 Thus, in order to prove  Proposition \ref{prop: 1} and Proposition \ref{prop: area} it suffices to find piecewise spherical metrics  $Y$ on $\mathbb S^2$ with at most one singularity of total angle larger than $2\pi$ and 
arbitrary large area, respectively,  diameter arbitrary close to $2\pi$.    

\begin{proof}[Proof of Proposition \ref{prop: area}]
Consider an interval $I=[a,b]$ of large  length $N$. Let $Z$ be the spherical join of a point $p$ and $I$.   The space $Z$ is topologically a closed disc and  it has curvature one in the interior. The boundary of $Z$ is built by two 
geodesics $pa$ and $pb$ of length $\pi /2$
and by the local geodesic $I$.  The angle at $a$ and $b$ is $\frac \pi 2$, the  total angle at $p$ equals $N$.  The area of $Z$ equals $N$.   
 
Consider the doubling $Y$ of $Z$ along the boundary. Then $Y$ is a piecewise spherical metric on the $2$-sphere, with 3 singularities of total angles $\pi, \pi, 2N$ and with total area $2N$. Due to the consideration preceding the proof, this suffices for the conclusion. 
\end{proof}

\begin{proof}[Proof of Proposition \ref{prop: 1}]
Fix  $\varepsilon <\frac \pi 2$. Consider a triangle $D=pxy$ in the round sphere $\mathbb S^2$ with  $px$ of length $\varepsilon$, with $\angle pxy =\frac \pi 2$ and with the length of $xy$ equal to $\pi-\varepsilon$.  Then $\angle pyx < \frac \pi 2 < \angle ypx$.

Consider another isometric  copy $D'=pxy'$ of the triangle and glue $D$ and $D'$ along the common  side $px$.  The arising space $Z$ is homeomorphic to a closed disc. It has constant curvature $1$ in the interior. The boundary is built by 4 geodesics $py$, $py'$,
$yx$ and $y'x$.  The angle at $x$ equals $\pi$, the angles at $y$ and $y'$ are smaller than $\pi$, the angle at $p$ is larger than $\pi$.  The diameter of $Z$ is at least twice the distance of $y$ and $y'$ which is larger than $2\pi- 4\varepsilon$. 

 The doubling $Y$ of $Z$ along the boundary  $\partial Z$ is homeomorphic to $\mathbb S^2$ and has diameter at least $2\pi -4\varepsilon$.
Moreover, $Y$ has piecewise constant curvature  $1$ and at exactly one singularity $p$
the  total angle is larger than $2\pi$.
Due to the consideration preceding the proof,  this  suffices for the conclusion, since $\varepsilon$ can be chosen arbitrary small.
\end{proof}

As a consequence we provide the following  negative answer to  
\cite[Question 8.7]{Liouv}.  We refer to the  discussion in \cite[Section 7]{Liouv} for motivation and relation with the Levy--Gromov inequality.

\begin{cor} \label{prop: 3}
For any $\epsilon>0$ there exist a smooth  Jordan curve $\Gamma$ in $\R^2$ bounding a Jordan domain $\Omega$   and
  a smooth $u:\R^2 \to \R$ satisfying  \eqref{eq: Liouv+},  such that 
$\int _{R^2} e^{2u} <\infty$ and the following holds true:
$$(\int _{\Gamma}  e^u)^2 \leq \varepsilon \cdot \int _{\Omega} e^{2u}  \cdot \int _{\R^2\setminus \Omega}  e^{2u} \;.$$
\end{cor}

\begin{proof}
The construction in the proof of Proposition \ref{prop: area} provides conformal metrics $X^u$
on $\R^2$ with curvature $\geq 1$ and arbitrary large but finite area $A=A(u)$. Moreover, by construction, any of this metric spaces $X^u$ contains a metric ball $\Omega$ of radius $r=\frac \pi {10}$ in the round sphere $\mathbb S^2$.  Let $l_0$  and $A_0$ denote the length of $\partial \Omega $, respectively   the area of $\Omega$  (both quantities measured in $X^{u}$, hence in $\mathbb S^2$).

   Then the right hand side
$(\int _\Gamma e ^u  )^2$ 
of the claimed inequality is just $l_0^2$ while  the factors on the right hand side are $A_0$ and $A-A_0$ respectively.
Thus, choosing $u$ such that the area  $A=\mathcal A (X^u)$  satisfies
$$A\geq A_0 + \frac {l_0} {\sqrt \varepsilon} \;,$$
we finish the proof.
\end{proof}

\section{Planes of finite area}

The next argument is  contained in  the proof of 
\cite[Theorem 1.4]{Liouv}.
\begin{proof}[Proof of Lemma  \ref{prop: incompl}]
The  space $X$ is a length space, hence so is  the completion $\hat X$, \cite[p. 43]{BBI01}. More precisely,
for any $x\in \hat X\setminus X$ there exists a curve of finite length $\gamma _x: [0,a) \to X$,  such that in $\hat X$ we have 
$$\lim _{t\to a} \gamma _x (t)= x \;. $$

Assume that we have two different points $x,y \in \hat X\setminus X$. Denote by $ \varepsilon>0$ the distance
between $x$ and $y$. 
Consider  curves $\gamma_x, \gamma _y$  of finite length in $X$ converging to $x$ and $y$,  as above.  By changing the starting points, we may assume that $\gamma _x$ and $\gamma _y$ have length smaller than $\frac \varepsilon 4$.   In order to obtain a contradiction, it suffices to find points on $\gamma_x$ and $\gamma _y$ with distance less than $\frac \varepsilon 4$ from each other.

Our space $X$ is the plane $\R^2$ with  the Euclidean metric changed by the conformal factor $e^{2u}$.  Denote by $\eta _r$ the Euclidean circle around $0$ of radius $r$.
We express the finiteness of the area in polar coordinates and obtain by the Hoelder inequality
$$\infty >\mathcal A(X) =\int _{\R^2} e^{2u}  = \int _0 ^{\infty}  (\int _{\eta _r}e ^{2u})\,  dr  \geq \int _0 ^{\infty} \frac 1 {2\pi r} (\int _{\eta _r} e^u)^2 \,dr\;.$$  

The length of $\eta _r$ in the metric space $X$ is  $\int _{\eta_r} e^u$.  Therefore, 
we find  a sequence  $r_i \to \infty$  such that the length
of $\eta _{r_i}$ is smaller than $\frac \varepsilon 4$.

Since the curves $\gamma _x$ and $\gamma _y$ do not have limit points in $X$,  both curves  run to infinity in $\R^2$. Hence they both  intersect $\eta_ r$, for all sufficiently large $r$.  Thus, for sufficiently large $r_i$ as above,  we find  points in  the intersection of $\gamma _x$ and $\gamma _y$ with $\eta _{r_i}$. The distance between these intersection points in $X$ is  less than $\frac \varepsilon 4$, in contradiction to our assumption. Hence, $\hat X \setminus X$  contains at most one point.
\end{proof}

We are going to explain  that $\hat X$ does not need to be  locally compact at the
point $\{p \} = \hat X \setminus X$.

\begin{ex} \label{ex: non}
Consider the round sphere $\mathbb S^2$ with north pole $p$. Take a sequence  $U_j$ of small metric balls     centered   on a fixed meridian starting at $p$.
 We choose the metric balls pairwise disjoint, not containing $p$, but converging to $p$.
Change the metric conformally on $\mathbb S^2\setminus \{  p\}$ in the following way.
The conformal factor is constantly one outside the union of all $U_j$.  The subset $U_j$ has after the conformal change diameter approximately $1$ and area approximately 
$\frac 1 {j^2}$, thus $U_j$ becomes a long  and very thin finger sticking out of the sphere.  The new metric on $\mathbb S^2 \setminus p$ is
conformally equivalent to $\R^2$, it  has finite area and diameter. Moreover, it has infinitely many points with pairwise distances in the interval $[2,3]$. Hence, the completion $\hat X$ cannot be locally compact by the theorem of Hopf--Rinow. 
\end{ex}

The next example  shows that even if $\hat X$ is compact and $X$ has curvature at least $1$, the curvature does not need to be integrable and the area growth at the singularity $p=\hat X\setminus X$ can be superquadratic.   

\begin{ex} \label{ex: Romney}
Consider the metric on $\R^2$ with conformal factor $e^{-\frac 2 {|z|}} \cdot |z|^{-4}$ as in
\cite[Section 5.1]{Romney}, \cite[Section 4.1]{CR}.  The area growth of this metric space $Y$  at $p=0$ is superquadratic, \cite[p. 19]{Romney}. Euclidean balls around $0$
are metric balls around $p=0$ in $Y$ and they are convex.  $Y$ is smooth outside of $0$ and direct computations reveal that the metric has positive curvature outside of $p$; moreover, the curvature converges to $\infty$ at  $p$.  Consider now a small closed
 ball $B$ around $0$ in $Y$ such that the curvature is larger than $1$ outside of $0=p$
and such that $\partial B$ has length $2\pi s < 2\pi$.  
Glue to $B$ along $\partial B$  a round hemisphere of radius $s$. By the gluing theorem (for instance, \cite{Petruning}), the arising sphere has  curvature $> 1$ outside the singularity  $0$. 
 Smoothing the metric along $\partial B$, (see for instance, \cite{Itoh}),
we obtain a smooth metric $\hat X$ on  $\mathbb S^2$, which has curvature $\geq 1$ everywhere outside a single point $p$ and that around $p$ the metric is isometric to $Y$.   By construction (and the uniformization theorem), the metric on 
$\hat X \setminus \{ p\}$ is conformally equivalent to $\R^2$. 
\end{ex}

 It seems possible but technically more involved  to construct an example of a conformal plane 
$X=X^u$ of curvature $\geq 1$ and finite area, such that the diameter of $X$ is $2\pi$ (thus  strengthening  Proposition \ref{prop: 1}). In such an example the completion $\hat X$ has to be non-compact.

\section{Planes of finite area and curvature}
\subsection{Integral bound} \label{subsec: int} If the conformal plane $X=X^u$ has finite total curvature we can control the geometry at infinity much better:

\begin{proof}[Proof of Theorem \ref{thm: main}]
First assume that $X=X^u$ is complete.  Then
the curvature estimate $\mathcal K(X) \leq 2\pi$ is a classical theorem of Cohn-Vossen,
\cite[Satz 6]{CV}, valid also for complete  planes of infinite area.
 Given that the area is finite,
the equality $\mathcal K(X)=2\pi$  is proven  in \cite[Corollary]{Shiohama}.

Finally, due to \cite[Korollar]{Huber}
(or, alternatively, \cite[Satz 3]{Huber})
if  $X$ is complete  then
 the function  $u$  extends to  a $\delta$-subharmonic function on $\mathbb S^2$, once $\R ^2$ is identified with $\mathbb S^2$ without a point by a conformal transformation.

From now on we assume that $X$ is not complete. We consider the completion $\hat X$ and  let $p$ be the unique point in $\hat X \setminus X$, Lemma \ref{prop: incompl}.

 First, we claim that $\hat X$ is compact.  Otherwise, we find some $\varepsilon >0$ and infinitely many points $x_i \in X$ with pairwise distance larger than  $2\varepsilon$. Removing at most one point, we can assume that the distance of any $x_i$ to $p$ is larger than  $\varepsilon$.  Then the  closed balls $\bar B_{\varepsilon} (x_i)$ are pairwise disjoint and compact. Moreover, removing finitely many $x_i$ and using the finiteness of 
total curvature, we may assume  that the total positive curvature of any $\bar B_{\varepsilon} (x_i)$ is at most $\pi$.  Then, for any $i$, we can estimate the area of the ball as
  $$\mathcal A(B_{\varepsilon} (x_i) \geq \frac \pi 2 \cdot \varepsilon ^2\,,$$
due to \cite[Proposition 3.2]{Shioya-conv}, \cite[Theorem 9.1]{Reshetnyak-GeomIV}. Thus, the finiteness of $\mathcal A(X)$ contradicts the disjointness of the balls $\bar B_{\varepsilon} (x_i)$.

Therefore, $\hat X$ is compact. Due to the uniqueness of the one-point-compactification, $\hat X$ is homeomorphic to $\mathbb S^2$.

In order to prove that $\hat X$ is a surface with bounded integral curvature we present the metric on $\hat X$ as a limit of metrics with a uniform integral  bound on curvature, as in \cite[Section 8.4]{Reshetnyak-GeomIV}. 

We claim that there exists a sequence of simple closed curves $\Gamma _j$ in $X$,
 such that for the Jordan domains $p\in O_j$ in $\hat X$ of $\Gamma _j$ 
the following holds true: 
The closure $\bar O_j =O_j \cup \Gamma _j$ is convex in $\hat X$;  the diameter of $\bar O_j$ and the length of
$\Gamma _j$ are at most  $\frac 1 j$.

 Note, that any such $\Gamma _j$ would be of bounded turn and the variation from the side of $X\setminus O_j$ (thus the mean curvature) would be   non-positive, by convexity of $O_j$, cf. \cite[Theorem 8.1.3]{Reshetnyak-GeomIV}.  Moreover, by the Gauss--Bonnet formula and the bound on the total curvature of $X$, the total curvature of $\Gamma_j$ would be  uniformly bounded.

 Once such $\Gamma _j$ are found, we would cut out $O_j$ and replace it by the round hemisphere $\hat O_j$ with boundary of length $\ell (\Gamma _j)$.  The arising space $\hat X_j$ is a 
sphere with uniformly bounded integral curvature, \cite[Theorems 8.3.1,  8.3.2]{Reshetnyak-GeomIV}.  Moreover, identifying $\hat O_j$ with $O_j$ by any homeomorphism fixing $\Gamma _j$, we obtain a convergence of $\hat X_j$ to $\hat X$ in the sense   of \cite[Section 8.4]{Reshetnyak-GeomIV}. Thus, $\hat X$ would be of  integrally  bounded
curvature, \cite[Theorem 8.4.5]{Reshetnyak-GeomIV}.

 It remains to find the required curves $\Gamma_j$.  In order to find them, we fix $j$
and set $\delta =\frac 1 {10 j}$.
Consider the open ball $U=B_{\delta} (p)$.  We find an index $i$, such that  for all $k\geq i$,  the curves $\eta _k :=\eta _{r_k}$ constructed in the proof of Lemma  \ref{prop: incompl} have length $\ell (\eta _{r_k} ) <\delta$. 
  By construction, the Jordan domains $p\in U_k$ of $\eta _k$ are nested and their intersection consists of the point $p$ only. Choosing $i$ large enough, we may assume in addition, that $U_k\subset U$, for all $k\geq i$.

We fix this $\eta_i$.  By compactness and local contractibility, there is  some 
$\varepsilon >0$ such that no closed curve in $\bar U_i$ of length  at most $ \varepsilon$ can intersect $\eta _i$  and be homotopic to $\eta _i$ within  the punctured disc $\bar U_j \setminus \{p \}$.      We now find some $k>i$ such that $\ell (\eta_k ) <\varepsilon $.

In the compact annulus  $A$ bounded by $\eta_k$ and $\eta_i$ in $X$ we find a 
shortest non-contractible curve $\gamma$. This $\gamma$ is automatically simple closed. By the choice of $k$, this curve $\gamma$   has length at most $\epsilon$, and by the choice of $\epsilon$, the curve $\gamma$ does not intersect $\eta _i$.  The Jordan domain $p\in V$  of this curve
is contained in $U$, thus has diameter  at most $2\delta$. If $V$ were not convex, then 
two points on $\gamma$ could be connected within $A$ by a shorter curve. But this would contradict the minimal property of $\gamma$.    This finishes the construction of $\gamma =\gamma _j$  and, therefore, of the statement that $\hat X$ has   bounded integral curvature.

The final statement that the metric of $\hat X$ is conformal to the round metric on the sphere including $p$ is a direct consequence 
of the uniformization for such surfaces, \cite[Section 7]{Reshetnyak-GeomIV},  \cite{Troyanov}.
\end{proof}

\begin{rem}
We have used some geometric arguments in the non-complete case in the proof above.
Possibly, a more analytic proof of the statement using the full strength of \cite[Satz 3]{Huber} can be found.
\end{rem}

Some additional comments on the structure of $\hat X$  near $p=\hat X\setminus X$, in case that $X$ is not complete in Theorem \ref{thm: main}:

  Consider  the curvature measure 
$\hat {\mathcal K}$ on the sphere $\hat X$ with bounded integral curvature,  \cite[Chapter 5]{AZ}, \cite{Troyanov}.
This is a signed measure satisfying $\hat {\mathcal K} (\hat X)=4\pi$ by the Gauss--Bonnet theorem \cite[p. 20]{Troyanov}.  On the regular part $X$, the signed measure $ \hat {\mathcal K}$ equals $K\cdot \mathcal H^2$, where
$K$ is the Gaussian curvature.   Thus, $\hat  {\mathcal K} (\{p \})= 4\pi - \mathcal K (X)$.
 On the other hand, $\hat {\mathcal K} (\{p\})$ equals $2\pi -\theta$, where $\theta$ is the \emph{total angle} at 
the point $p$
 \cite[Lemma 8.1.1]{Reshetnyak-GeomIV}. Moreover, again by  
 \cite[Lemma 8.1.1]{Reshetnyak-GeomIV} and the coarea formula (or using    \cite[Theorem  9.10]{Reshetnyak-GeomIV})
$$\theta = \lim_{r\to 0} \frac {\mathcal H^1 (\partial B_r(p))} {r} = 2 \lim_{r\to 0} \frac {\mathcal H^2 (B_r(x))} {r^2}\;.$$

\subsection{Smoothness at infinity}
We are going to provide 
\begin{proof}[Proof of Proposition \ref{prop: final}]
 We can apply Theorem \ref{thm: main}. Identifying $\R^2$ conformally with the complement of a point $p$ in $\mathbb S^2$, we obtain that the completion $\hat X$ is 
a sphere with curvature bounded in the integral sense of Alexandrov. The \emph{curvature measure} $\hat {\mathcal K}$ of $\hat X$ coincides on $X$  with the multiple 
$\hat {\mathcal K} = K\cdot \mathcal H^2$ of the canonical area measure $\mathcal H^2$.  By the Gauss-Bonnet theorem, $\mathcal {\hat K}(\hat X)= 4\pi$,. Thus, by assumption, $\hat {\mathcal K} (\{ p\})=0$.   Therefore, the equality $\hat {\mathcal K} = K\cdot \mathcal H^2$ is valid on all of $\hat X$.

Therefore, on all of $\hat X$, the metric is defined by a conformal change $e^{2\hat u} \cdot \delta _{\mathbb S^2}$ of the round metric on $\mathbb S^2$, such that the  spherical Laplacian of $\hat u$ is  a bounded function $K  + 1$.  Elliptic regularity implies that $\hat u$ is of class $\mathcal C^{1,\alpha}$ for any $\alpha <1$. Moreover,
if the curvature  $K:X=\R^2 \to \R$ extends as a $\beta$-Hoelder continuous function to $\mathbb S^2$ then $\hat u$ is $\mathcal C^{2,\beta}$-Hoelder.

  An example of a conformal metric $e^{2 v} \cdot \delta _{\R^2}$ on a disc , which is  smooth outside the origin $p$, not 
$\mathcal C^{1,1}$ at the origin and,  such that the Lapalcian  $ \Delta v $ is continuous, is presented in 
\cite[p. 693]{Sabitov}.  This metric (restricted to a subdisc) can clearly be extended to a metric on the sphere, which  has continuous curvature but is not $\mathcal C^{1,1}$ in conformal coordinates.  This example finishes the proof.
\end{proof}

\bibliographystyle{alpha}
\bibliography{planes}

\end{document}